\documentclass{article}
\usepackage{spconf,amsmath,graphicx,color}

\usepackage[draft]{changes}

\usepackage{lipsum}

\usepackage{xcolor}
\usepackage{hyperref}
\usepackage{amsmath}
\usepackage{bm}
\usepackage{bbm}
\usepackage{float}
\usepackage{amsthm}
\usepackage{amssymb}
\usepackage{algorithm}
\usepackage{algorithmicx}
\usepackage{algpseudocode}
\usepackage{graphicx}
\usepackage{epstopdf}
\usepackage{enumerate}
\usepackage{enumitem}
\usepackage{array}

\newtheorem{theorem}{Theorem}[]
\newtheorem{corollary}{Corollary}[section]
\newtheorem{lemma}[]{Lemma}
\newtheorem{proposition}{Proposition}[section]
\newtheorem{assumption}{Assumption}[]

\newcolumntype{L}{>{\centering\arraybackslash}m{1.1cm}}

\usepackage[draft]{changes}
\title{Zeroth-Order Randomized Subspace Newton Methods}
\name{Erik Berglund, Sarit Khirirat, Xiaoyu Wang\thanks{This work was supported partially by the Swedish Foundation for Strategic Research (SSF) project SoPhy and the Wallenberg Artificial Intelligence,
Autonomous Systems and Software Program (WASP) funded by Knut
and Alice Wallenberg Foundation. Emails: \texttt{erbergl@kth.se, sarit@kth.se, wang10@kth.se}. © 2022 IEEE.  Personal use of this material is permitted. Permission from IEEE must be obtained for all other uses, in any current or future media, including reprinting/republishing this material for advertising or promotional purposes, creating new collective works, for resale or redistribution to servers or lists, or reuse of any copyrighted component of this work in other works.}}
\address{KTH Royal Institute of Technology}
\begin{document}
%
\maketitle
\begin{abstract}
Zeroth-order methods have become important tools for solving problems where we have access only to function evaluations. However, the zeroth-order methods only using gradient approximations are $n$ times slower than classical first-order methods for solving $n$-dimensional problems. To accelerate the convergence rate, this paper proposes the zeroth order randomized subspace Newton (ZO-RSN) method, which estimates projections of the gradient and Hessian by random sketching and finite differences. This allows us to compute the Newton step in a lower dimensional subspace, with small computational costs. We prove that ZO-RSN can attain lower iteration complexity than existing zeroth order methods for strongly convex problems. Our numerical experiments show that ZO-RSN can perform black-box attacks under a more restrictive limit on the number of function queries than the state-of-the-art Hessian-aware zeroth-order method.
\end{abstract}
\begin{keywords}
Zeroth-order optimization, sketching techniques, Newton-type method, adversarial black-box attacks, convolutional neural network. 
\end{keywords}
\section{Introduction}
\label{sec:intro}

Several applications in machine learning, signal processing and communication networks can often be cast into
optimization 
problems,
where gradients are difficult or even infeasible to compute. 
%
Popular application examples 
include optimal hyper-parameter tuning for learning models \cite{snoek2012practical,bergstra2011algorithms}, black-box adversarial attacks on neural network models \cite{chen2017zoo,papernot2017practical,hu2017generating,ilyas2018black} and sensor selection problems in smart grids or wireless networks \cite{liu2016sensor,hero2011sensor,liu2018zeroth}.
%
%
This motivates the study of the zeroth-order methods. A prominent type of zeroth order methods uses function value differences to estimate the gradients \cite[Section 3.4]{book}. 
However, these methods  
are much slower 
%
%
%
than classical gradient descent 
~\cite{Nesterov2017}, and also 
suffers from poor performance particularly for ill-conditioned problems. An alternative way to improve their performance is to incorporate the second order information into zeroth-order methods. 
However, computing the full Hessian matrix can heavily increase the number of function evaluations and make the Newton step hard to compute, 
especially for high-dimensional problems.
This necessitates us to approximate the Hessian matrix in a lower-dimensional subspace. 

Ye \emph{et al.}\cite{YeHaishan2018HZOf} developed  the Hessian-aware zeroth order (ZOHA) methods, which integrate Hessian information into zeroth-order methods. 
%
The power-iteration based method ZOHA-PW 
has a lower query complexity than the gradient-estimating method by 
\cite{Nesterov2017} when the eigenvalues of the Hessian decay sufficiently quickly. However, the power iteration method requires $O(n)$ function queries per iteration for $n$-dimensional problems, which is expensive when $n$ is large.
To decrease the query cost, they proposed the heuristic methods ZOHA-Gauss-DC and ZOHA-Diag-DC, which estimate the Hessian based on a limited number of random directions. However, no complexity bounds are provided for them. 

%
%

%

Another approach to reduce the times of computing Hessian information for high-dimensional problems
is to use randomized sketching techniques \cite{tang2017gradient,gower2019rsn,pilanci2017newton}. 
These sketching techniques 
construct lower dimensional sub-problems, which can be solved within small computation times,
%
and enable classical optimization algorithms to have better scalability.
For instance, a randomized subspace newton (RSN) method \cite{gower2019rsn} exploits the sketching techniques on the Newton method to solve the problems with very large dimension and to achieve accelerated convergence rate.
  
In this paper, we propose Hessian-based zeroth-order algorithms using sketching techniques for huge-dimensional problems, called zeroth-order RSN (ZO-RSN).
The methods exploit finite differences and sketching to approximate projections of the gradient and Hessian. 
%
%
%
We provide complexity bounds and prove that under certain conditions ZO-RSN attains lower query complexity than existing zeroth-order algorithms for strongly convex problems. 
Finally, our experiments with black-box attack problems on a convolutional neural network 
show that ZO-RSN has an overall competitive performance and higher success rate, compared to the ZOHA-Gauss-DC method in \cite{YeHaishan2018HZOf}.

\subsection{Notation}
 For $x\in\mathbb{R}^n$ and $M \succ 0$, $\|x\|_2$ and $\|x\|_\infty$ are the $\ell_2$ and $\ell_\infty$ norm, respectively, and $\|x\|_M^2 = x^T M x$. 
%
%
Given the sketching matrix $S\in\mathbb{R}^{n\times m}$, $s_1,s_2,\ldots,s_m\in\mathbb{R}^n$ are its columns. 
For $f:\mathbb{R}^n\rightarrow\mathbb{R}$, $g(x)=\nabla f(x)$ and $H(x)= \nabla^2 f(x)$ are its gradient and Hessian. 
The function $f(x)$ is $L$-Lipschitz continuous if there exists a positive constant $L$ such that
$
\| f(y)  - f(x)\|_2 \leq L\|y-x\|_2 \text{ for all } x,y\in\mathbb{R}^n,
$
and $\mu$-strongly convex if there exists a positive constant $\mu$ such that
$
f(y) \geq f(x) + \langle \nabla f(x) , y-x \rangle + ({\mu}/{2})\| y-x\|^2_2 \text{ for all } x,y\in\mathbb{R}^n.
$
We also state that the differentiable function $f(x)$ is $L_s$-smooth if its gradient $g(x)$ is $L_s$-Lipschitz continuous. Finally, for any $y \in \mathbb{R}^{n}$, $\Delta_y f(x) = f(x+y)-f(x)$.

\section{Problem Formulation}\label{sec:problem_form}

We consider the  unconstrained optimization problem 
\begin{equation}
\begin{aligned}
& \underset{x \in \mathbb{R}^n}{\text{minimize}}
& & f(x),
\end{aligned} \label{eqn:unconstrained_problem}
\end{equation}
where the dimension $n$ could be very large. 
Here, $f(x)$ is a three times differentiable and $\mu$-strongly convex function,
which is bounded from below and has its minimum value $f^{\ast}$ at the point $x^\ast$. 
\textcolor{black}{$g(x)$ and $H(x)$ are also $L_1$- and $L_2$-Lipschitz continuous.}
To facilitate the analysis, we further make the following standard assumption on $f(x)$. 
\begin{assumption}[\cite{gower2019rsn,praneeth2018global}]
\label{asm:rel_cvx_smth}
There exists $\hat L \geq \hat \mu > 0$ such that for any $x,y\in \mathbb{R}^n$:
\begin{equation}\label{eqn:lhatbound}
    f(x) \leq f(y) + g(y)^T(x-y) + ({\hat L}/{2}) \|x-y\|_{H(y)}^2, 
\end{equation}
\begin{equation}\label{eqn:muhatbound}
    f(x) \geq f(y) + g(y)^T(x-y) + ({\hat \mu}/{2}) \|x-y\|_{H(y)}^2.
\end{equation}
\end{assumption}
Assumption \ref{asm:rel_cvx_smth} states the smoothness and strong convexity of $f(x)$ under the norm weighted by its Hessian $\| \cdot \|_{H(x)}$.
%
Also, the $\hat L$-relative smoothness and $\hat \mu$-relative convexity exist as a result of the $L_1$-smoothness and $\mu$-strong convexity  assumption on $f(x)$, as shown below: 
\begin{proposition}[\cite{gower2019rsn,praneeth2018global}]\label{prop:c_stab}
    A function $f(x)$ is $c$-stable on a domain $D$ if $\forall y,z  \in D$, $\|z - y\|_{H(y)}^2$ and there exists a constant $c \geq 1$ such that
   $
        c = {\|z - y\|_{H(z)}^2}/{\|z - y\|_{H(y)}^2}.
  $
    If $f(x)$ is $\mu$-strongly convex and $L_1$-smooth, then $f$ is $(L_1/\mu)$-stable. Furthermore, if $f(x)$ is $c$-stable, then Assumption~\ref{asm:rel_cvx_smth}
holds with $\hat L \leq c$ and $\hat \mu \geq {1}/{c}$.
\end{proposition}
%

\subsection{RSN Methods} 
The randomized subspace Newton (RSN) method \cite{gower2019rsn} is a popular inexact Newton method for solving huge-dimensional problems. 
This method solves an exact Newton system restricted to a random subspace. 
Given a fixed step-size  $\gamma >0$ and an initial point $x_0\in\mathbb{R}^d$,
%
the iterate $x_k$ of the RSN method is updated via: \vspace{-0.1cm}
\begin{equation}\label{eqn:updaterule}
    x_{k+1} = x_k + \gamma S_k \lambda_k, ~ 
    ~ S_k^TH(x_k)S_k \lambda_k = - S_k^T g(x_k),
\end{equation}
where 
$S_k\in\mathbb{R}^{n\times m}$ stores $m$ vectors that span the randomly selected subspace of $\mathbb{R}^n$. 
The next lemma characterizes the decrease in the function value from the ZO-RSN method \eqref{eqn:updaterule}.
%
\begin{lemma}\label{lemma:RSNDescent}
Consider the RSN method \eqref{eqn:updaterule} for solving Problem \eqref{eqn:unconstrained_problem}. If $\gamma \leq 1/\hat L$, then 
\begin{equation}\label{eqn:RSNbound}
f(x_{k+1}) \leq f(x_k) - ({\gamma}/{2}) \|g(x_k)\|^2_{S_k(S_k^TH(x_k)S_k)^{\dagger}S_k^T}.
\end{equation}
\end{lemma}
This descent lemma for the RSN method can be used to prove its linear convergence toward the exact optimum \cite{gower2019rsn}.
%
%
%
%
%
Furthermore, to implement the RSN method $S_k^TH(x_k)S_k$ and $S_k^T g(x_k)$ are computed efficiently by various sketching techniques such as sub-Gaussian sketches, randomized orthonormal system sketches, random sampling sketches and the Iterative Hessian Sketch \cite{10.5555/2946645.2946698} as well as the fast Johnson-Lindenstrauss sketch for problems with the appropriate structure \cite{doi:10.1137/060673096}. 
These sketching techniques allow for computing $\lambda_k$ with very small linear equation systems. If $m \ll n$, then $\lambda_k$ in Eq. \eqref{eqn:updaterule} can be solved quickly by inverting $S_k^TH(x_k)S_k \in \mathbb{R}^{m\times m}$. 
%

\section{Zeroth-order RSN Methods}
In this section, we introduce the zeroth-order randomized subspace Newton (ZO-RSN) method, which builds on the RSN method. 
The iterate $x_k$ of the ZO-RSN algorithm is updated according to:
\begin{align}
    x_{k+1} = x_k + \gamma  S_k \tilde \lambda_k, ~~ \text{and} ~~
    \tilde{H}_{S_k}(x_k)   \tilde{\lambda}_k = - \tilde{g}_{S_k}(x_k).\label{eqn:updaterule_ZO} 
\end{align}
Here $\tilde{g}_{S_k} (x_k)$ and  $\tilde{H}_{S_k} (x_k)$ are approximations of the sketched gradient and Hessian respectively. For a positive scalar $\alpha$, they can be computed via: 
%
\begin{align*}
     [\tilde{g}_{S_k}(x_k)]_i 
      &:= {\Delta_{\alpha s_{i,k}}f(x_k)}/{\alpha} \approx s_{i,k}^T g(x_k), 
      \intertext{and }
      [\Tilde{H}_{S_k}(x_k)]_{i,j} &: = \Delta_{\alpha s_{i,k}}\Delta_{\alpha s_{j,k}} f(x_k)/{\alpha^2}  \approx s_{i,k}^T H(x_k) s_{j,k}, 
\end{align*}
for all $i=1,\ldots,m$.
Similarly to Lemma \ref{lemma:RSNDescent}, the ZO-RSN method can be proved to achieve the following bound: \begin{equation}\label{eqn:RSNbound_FD}
  \hspace{-0.5em}  f(x_{k+1}) \leq f(x_k) - \frac{\gamma}{2} \|g(x_k)\|^2_{S_k(S_k^TH(x_k)S_k)^{\dagger}S_k^T} + O(\alpha).
\end{equation}
This ensures function value improvement in Eq. \eqref{eqn:RSNbound_FD} if $\alpha$ is sufficiently small and $\tilde{H}_{S_k}(x_k)$ is positive definite. In fact, we can ensure that positive definiteness of $\tilde{H}_{S_k}(x_k)$ follows from $\alpha$ being small enough if we choose $S_k$ appropriately.
\begin{lemma}\label{lemma:PD_H_tilde}
%
If $S_k^TS_k = I$ and  $\|\tilde{H}_{S_k}(x_k) - S_k^T H(x_k) S_k \|_2 < \mu$, then $\tilde{H}_{S_k}(x_k) \succ 0$.
\end{lemma}
%
%
Based on this lemma, we set $S_k^TS_k = I$ to ensure that $\tilde{H}_{S_k}(x_k) \succ 0$. We also require $\mathbb{E}[S_kS_k^T] \succ 0$ so that the approximate sketching does not leave out any directions throughout every iteration.
This requirement can be easily satisfied if $s_{1,k},...,s_{m,k}$ are sampled from unit coordinate directions without replacement.


\section{Theoretical results}
\label{sec:convergence_analysis}

We now provide a complexity bound for ZO-RSN methods. 
%
%
%
%
%
\begin{theorem}\label{thm:Convergence}
Let the sketching matrix $S_k\in\mathbb{R}^{n\times m}$ satisfy $S_k^T S_k = I$ and $\mathbb{E}_{S_k \sim D}[S_kS_k^T] \succ 0$, and define $G(x) = \mathbb{E}_{S_k \sim \mathcal{D}}[S_k(S_k^TH(x)S_k)^{-1}S_k^T]$, \vspace{-0.1cm}
    \begin{equation*}
      \rho(x)=  \min_{v\in \mathbb{R}^n} \dfrac{v^TH(x)^{\frac{1}{2}}G(x)H(x)^{\frac{1}{2}}v}{\|v\|_2^2} ~~~ \text{and}  ~~~  \rho = \min_{x\in\mathbb{R}^n} \rho(x).
      \end{equation*}
 Given $\varepsilon >0$ and $\delta\in (0,1)$, consider the ZO-RSN method \eqref{eqn:updaterule_ZO} for Problem \eqref{eqn:unconstrained_problem}.
%
If $\gamma \leq 1/\hat L$ and $\alpha \leq 0.3 \mu/(m L_2)$ is small enough that\vspace{-0.2cm} 
    \begin{align*}
    \dfrac{\alpha(C_1+C_2\alpha)}{\rho \hat \mu \gamma -\alpha C_1 - \alpha^2 C_3} \leq \delta \varepsilon \quad \text{and} 
    \quad \alpha C_1 + \alpha^2 C_3 < \rho \hat \mu \gamma,
    \end{align*}
    then 
    we can achieve $\mathbb{E}[f(x_{k})-f^{\ast}]\leq \varepsilon$  after 
    \begin{equation*}
    k \geq \left \lceil {\log \left (  \frac{f(x_0)-f^{\ast}}{(1-\delta) \varepsilon} \right )}\middle/{\log \left( \frac{1}{1 - \rho \hat \mu \gamma + \alpha C_1 + \alpha^2 C_3}\right)} \right \rceil 
\end{equation*}
iterations  where $C_1 = \gamma(\sqrt{m}L + B)/(2\mu)$, $C_2 = \gamma L_1^2 [m + \sqrt{m}(1+B)]/(2\mu^2)$, $C_3 = \gamma L_1 [ \sqrt{m}L_1(1+B) + B(2+B) ] /(2\mu^2)$ and $B = 10mL_2/(3\mu)$.
\end{theorem}

Theorem~\ref{thm:Convergence} establishes a global, linear convergence for the ZO-RSN method toward an $\varepsilon$-accurate solution. 
%
The worst-case 
iteration complexity 
can be upper bounded as \vspace{-0.1cm}
\begin{equation}\label{eqn:weaker_bound}
   k \geq \left \lceil \beta_1 \log \left (  {[f(x_0)-f^{\ast}]}/{[(1-\delta) \varepsilon}] \right ) \right \rceil.
\end{equation}
where $\beta_1 = 1/(\rho \hat \mu \gamma -\alpha C_1 - \alpha^2 C_3)$.
%
We can recover the convergence complexity for the RSN method~\cite{gower2019rsn} if $\alpha$ and $\delta$ approach zero.
Furthermore, 
by choosing $S_k$ properly, the iteration complexity for the ZO-RSN method in Eq.~\eqref{eqn:weaker_bound} can be lower than the complexities for existing zeroth-order methods.
We show this with the following corollary: 
\begin{corollary}\label{cor:specialcase}
%
%
Suppose all the conditions of Theorem~\ref{thm:Convergence} hold. If the columns of $S_k$ are chosen randomly without replacement from a basis of orthonormal eigenvectors of $H(x_k)$, step-size $\gamma = {1}/{\hat L}$, and  $\alpha = (\sqrt{C_1^2/4 + (1-\sigma)\rho \hat \mu \gamma} -{C_1}/{2} )/C_2$ for some $\sigma \in (0,1)$, 
then $\rho = m/n$ and hence 
to achieve $\mathbb{E}[f(x_{k})-f^{\ast}]\leq \varepsilon$, we need \vspace{-0.2cm}
    \begin{equation}\label{eqn:special_bound}
   k \geq \left \lceil ({n \hat L}/{[\sigma m \hat \mu]}) \log \left (  {[f(x_0)-f^{\ast}]}\middle/[{(1-\delta) \varepsilon}] \right ) \right \rceil.
\end{equation}
\end{corollary}

Corollary~\ref{cor:specialcase} shows that the iteration complexity of the ZO-RSN methods depends on the subspace dimension $m$, the problem dimension $n$ and other parameters $\hat\mu,\hat L$.
Since the ZO-RSN methods need $m(m+1)/2$ function queries per iteration, we can obtain the total query complexity by multiplying Eq.~\eqref{eqn:special_bound} with this factor. 
%
%

Now, we compare the complexity bounds for the ZO-RSN methods against the Hessian-aware zeroth-order method using the power iteration (ZOHA-PW) \cite{YeHaishan2018HZOf}, which previously has been compared favourably to the zeroth-order method in \cite{Nesterov2017}. Since the ZOHA-PW  method also generates multiple random directions, here $m$ refers to the number of the generated directions. For $\mu$-strongly convex problems, the iteration complexity of ZOHA-PW is 
%
%
%
%
%
\vspace{-0.2cm}
\begin{align}\label{eqn:IterationComplexity_ZO_2}
   k \geq \left \lceil \beta_2 \log \left( {[f(x_0)-f^{\ast}]}/{[(1-\hat{\delta}) \varepsilon]}\right)  \right \rceil, 
\end{align}
where $\beta_2 = {64(n+2)(\mu+10 \lambda_{s+1})}/{(\mu m)}$,  
$\lambda_{s+1}$ is an upper bound on the Hessian's $(s+1)^\text{th}$ largest eigenvalue and $\hat \delta$ is a free parameter which is similar to $\delta$ in Eq.~\eqref{eqn:special_bound}. Disregarding the function evaluations required to implement the power method, the total query complexity for ZOHA-PW is $2m$ times its iteration complexity. Consider the problem of minimizing a quadratic function. Then, $\hat L = \hat \mu= 1$. 
If $\delta$, $\hat \delta$ and $m$ all are set to be equal for both methods, and also $\sigma=0.5$, then the speedup in iteration complexity from using ZO-RSN instead of ZOHA-PW is \vspace{-0.2cm}
$$32\left(1+{2}/{n}\right)\left(1 + {10 \lambda_{s+1}}/{\mu}\right).$$
ZO-RSN is thus faster than ZOHA-PW by more than two orders of magnitude in iteration complexity, even for well-conditioned problems (when $\lambda_{s+1}/\mu$ is close to one). 
If function queries can be performed efficiently in parallel, then  ZO-RSN has significantly lower run-time than  ZOHA-PW. 
We can also prove that the speedup in query complexity for ZO-RSN compared to ZOHA-PW is \vspace{-0.2cm}
$$[{128\left(1+{2}/{n}\right)}\left(1 + {10 \lambda_{s+1}}/{\mu}\right)]/(m+1).$$
Thus, as long as $m < {128}\left(1 + {10 \lambda_{s+1}}/{\mu}\right)-1$, the query complexity will be lower for ZO-RSN.

\section{Numerical experiments}
\label{sec:numerical_results}

We compare the performance of ZO-RSN against the existing Hessian-aware zeroth methods called ZOHA-Gauss-DC \cite{YeHaishan2018HZOf} that uses a descent-checking procedure to increase an attack success rate, and  approximates Hessian according to 
\vspace{-0.2cm}
$$
\tilde H = (2\alpha^2b)^{-1} \sum_{i=1}^b | \Delta_{\alpha u_i}f(x) - \Delta_{\alpha u_i}f(x - \alpha u_i) | u_iu_i^T + \lambda I_d,
$$
where $\lambda$ is a positive constant and $u_1,\ldots, u_b$ are the vectors generated from the Gaussian distribution with zero mean and unit variance.  
In particular, we evaluate both methods on training un-targeted black box adversarial attacks over the  \texttt{MNIST} data set  \cite{CarliniNicholas2017TEtR,YeHaishan2018HZOf}. 
%
These attacks are carried out against the trained convolutional neural network (CNN) model described in \cite{YeHaishan2018HZOf}[Section 5.2].  
For each example $x^{nat}_{i}$ in the test set, the optimizer aims to generate an adversarial example $x_i$ which differs from $x^{nat}_i$ by at most $\epsilon$ in $\ell_{\infty}$ norm, while being classified differently with sufficient confidence. This is done by minimizing the following function \cite{CarliniNicholas2017TEtR}: \vspace{-0.2cm}
\begin{equation}
    f(x) = \max \left \{ \max_{i \neq l} \log [Z(x)]_i - \log [Z(x)]_l, - \omega \right\}.
\end{equation}
Here, $[Z(x)]_i$ represents the probability of an input $x$ belonging to class $i$ according to the trained neural network.

Since the problem is constrained and does not have guarantees for $\mu$-strong convexity or $L_1$-smoothness, we need to modify the ZO-RSN algorithm. Firstly, we artificially ensure positive definiteness and boundedness of $\tilde{H}_{S_k}(x_k)$ by applying the operator $\Pi_{[\lambda_{\min},\lambda_{\max}]}(\cdot)$ that projects its eigenvalues onto an interval $[\lambda_{\min},\lambda_{\max}]$ to get a modified matrix $\hat{H}_{S_k}(x_k)$. 
Secondly, we consider $\ell_\infty$-norm constraints by determining $\tilde{\lambda}_k$ that solves the following minimization problem
%
\vspace{-0.2cm}
\begin{equation}\label{eqn:QP_subproblem}
\begin{aligned}
& \underset{\lambda \in \mathbb{R}^m}{\text{minimize}}
& &  f(x_k) + \gamma \tilde{g}_{S_k}(x_k)^T\lambda + \frac{\gamma}{2}\|\lambda\|_{\hat{H}_{S_k}(x_k)}^2 \\
& \text{subject to} & -& \gamma S_k \lambda \leq x_k - x_i^{nat} -\mathbf{1} \epsilon \\ 
& & & \gamma S_k \lambda \leq \mathbf{1}\epsilon + x^{nat}_i - x_k.
\end{aligned} 
\end{equation}
This approach corresponds to using sequential quadratic programming (SQP) for nonlinear problems with linear constraints, but with the step to the next iterate being restricted to lie in a specific subspace.
To solve the auxiliary problem \eqref{eqn:QP_subproblem} quickly with a standard \texttt{cvxopt} solver \cite{andersen2021cvxopt}, 
%
we generate $S_k$ by choosing its columns to be unit coordinate vectors. 
This enables us to formulate the problem with only $m$ constraints.
This adapted ZO-RSN algorithm is called ZO-RSN-SQP.
%
%
Finally, we use the descent-checking technique corresponding to that for ZOHA-Gauss-DC. The full description of   ZO-RSN-SQP is given in Algorithm~\ref{alg:ZO-RSN-SQP_ICASSP}.

\begin{algorithm}\footnotesize
\caption{ZO-RSN-SQP for black-box attack}
\begin{algorithmic}
\item Initialize $x_0 \gets x^{nat}_i,\alpha,\gamma,m,m_{max}$
\For{$k = 0,1,...,k_{max}$}
    \State Generate $S_k = [s_{1,k},...,s_{m,k}]$
    \State Compute $\tilde{g}_{S_k}$ and  $\tilde{H}_{S_k}$
    \State $\hat{H}_{S_k} \gets \Pi_{[\lambda_{\min},\lambda_{\max}]}(\tilde{H}_{S_k})$
     \State $\tilde{\lambda}_k \gets \text{Solution to \eqref{eqn:QP_subproblem} with }\hat{H}_{S_k}(x_k) \text{ and } \tilde{g}_{S_k}(x_k)$
    \State $x_{trial} \gets x_k + S_k \tilde{\lambda}_k $
    \While{$f(x_{trial}) \geq f(x_k)$ and $\bar m < m_{max}$}
        \State  $\bar m \gets \bar m + 1$ 
        \State Generate $s_{\bar m,k} \text{ such that } [S_k,s_{\bar m,k}]^T[S_k,s_{\bar m,k}] = I$
        \State $S_k \gets [s_{1,k},...,s_{\bar m,k}]$
        \State $[\Tilde{g}_{S_k}(x_k)]_{\bar m} \gets {\Delta_{\alpha s_{i,k}}f(x_k)}/{\alpha}$
        \For{$j = 1,2,...,\bar m$}
            \State [$\Tilde{H}_{S_k}(x_k)]_{\bar m,j}
            \gets   \Delta_{\alpha s_{i,k}}\Delta_{\alpha s_{j,k}} f(x_k)/{\alpha^2} $
            \State $[\Tilde{H}_{S_k}(x_k)]_{j,\bar m} \gets [\Tilde{H}_{S_k}(x_k)]_{\bar m,j}$
        \EndFor
        \State $\hat{H}_{S_k} \gets \Pi_{[\lambda_{\min},\lambda_{\max}]}(\tilde{H}_{S_k})$
        \State  $\tilde{\lambda}_k \gets \text{Solution to \eqref{eqn:QP_subproblem} with }\hat{H}_{S_k}(x_k) \text{ and } \tilde{g}_{S_k}(x_k)$
        \State $x_{trial} \gets x_k + \gamma S_k \tilde{\lambda}_k$
    \EndWhile
    \If{$f(x_{trial}) \leq f(x_k)$}
        \State $x_{k+1} \gets x_{trial}$
    \Else 
        \State $x_{k+1} \gets x_k$
    \EndIf
\EndFor
\end{algorithmic}\label{alg:ZO-RSN-SQP_ICASSP}
\end{algorithm}

We trained the network model until its accuracy reached $98.84\%$, and also set $\alpha = 0.1, \gamma=1, m =3$ and $m_{\max}=20$ for ZO-RSN-SQP and the same parameters for ZOHA-Gauss-DC for the un-targeted black box attacks described in \cite{YeHaishan2018HZOf}.
In the experiments, we either ended a test run if the algorithm managed to find a point with function value at $\omega = -1$, or if the algorithm  called queried the neural network for a prediction 50000 times. We labelled the former result as a success and the latter result as a failure. 

\begin{table}[t]
    \centering
    \begin{tabular}{|c|c|c|c|c|} \hline
         Algorithm &  ZO-RSN-SQP & ZOHA-Gauss-DC \\ \hline 
         Success rate (\%) & {$\bf 100$}  & $95.33$  \\ \hline 
         Median queries & $2336$ & {$\bf 815$} \\ \hline 
         Mean queries & {$\bf 2510$} & $4164$ \\ \hline 
         Max queries & {$\bf 8239$}  & $50000$ \\ \hline
         $f_{est2000}-f^*$ & $1.94$ & $3.70\cdot 10^{-1}$ \\ \hline
         $f_{est4000}-f^*$ & $1.89 \cdot 10^{-1}$ & $1.86 \cdot 10^{-1}$\\ \hline
         $f_{est6000}-f^*$ & $2.42 \cdot 10^{-2}$ & $1.47 \cdot 10^{-1}$ \\ \hline
    \end{tabular}
    \caption{Comparison of $\ell_{\infty}$ norm based black-box attacks on a CNN model trained on the \texttt{MNIST} data.}
    \label{tab:results}
\end{table}

%
The results of our black box attack experiments were summarized in Table~\ref{tab:results}. 
Firstly, ZO-RSN-SQP has a more stable performance than ZOHA-Gauss-DC. Even though both algorithms implement the same decent checking technique, only ZO-RSN-SQP 
succeeds in the attacks for all cases. Secondly, the mean number of queries for ZO-RSN-SQP is lower than that for ZOHA-Gauss-DC. This results from a minority of the problems, where ZOHA-Gauss-DC requires a large number of queries to solve. In contrast, ZOHA-Gauss-DC has a lower median value than ZO-RSN-SQP. As ZO-RSN requires more function queries per iteration and subspace dimension than ZOHA-Gauss-DC, one can hypothesize this extra effort is worthwhile mainly for the harder-to-attack test examples. 
%
%


To investigate the speed of convergence, we also ran a separate experiment where we made estimates of the average objective value after 2000, 4000 and 6000 queries, $f_{est2000}$, $f_{est4000}$ , $f_{est6000}$, using the first 100 MNIST examples. The suboptimalities based on these results are also shown in Table~\ref{tab:results}. As we can see, ZOHA-Gauss-DC is initially faster, but ZO-RSN-SQP becomes more accurate towards the end.

\section{Conclusions}
\label{sec:conclusions}
We have proposed the ZO-RSN method, a Hessian-based zeroth-order method that approximates sketched gradients and Hessians by finite differences.
Our results display a lower iteration complexity of the ZO-RSN method than existing zeroth-order methods for strongly convex problems.
The experiments with un-targeted adversarial attacks on a CNN model illustrate
%
that the modified ZO-RSN method named ZO-RSN-SQP attains an overall competitive performance and a higher stability, compared to ZOHA-Gauss-DC.

%
%

\vfill\pagebreak

\bibliographystyle{IEEEbib}
\bibliography{refs}

\newpage

\appendix

\section{Proof of Lemma 1}
If $\gamma \leq 1/\hat L$, then from Eq. \eqref{eqn:lhatbound} with $x = x_{k+1}, y=x_k$
\begin{equation}\label{eqn:gammabound}
    f(x_{k+1}) \leq f(x_k) + g(x_k)^T(x_{k+1}-x_k) + \frac{1}{2\gamma} \|x_{k+1}-x_k\|_{H(x_k)}^2.
\end{equation}
Utilizing the updates from  Eq.  \eqref{eqn:updaterule} that  $x_{k+1} - x_k  = \gamma S_k \lambda_k$ and $\lambda_k = -(S_k^T H_k S_k)^\dagger S_k^T g(x_k)$, we complete the proof.

\section{Proof of Lemma 2}
If $S_k^TS_k = I$ and also $v \in \mathbb{R}^m$ has norm 1, then  $\mu \leq v^T S_k^T H(x_k) S_k v_k$. This condition implies that $S_k^TH(x_k)S_k$ is positive definite and its lowest eigenvalue is bounded by $\mu$. Then, 
$\tilde{H}_{S_k}(x_k)\succ 0$ iff $\tilde{H}_{S_k}(x_k) (S_k^T H(x_k) S_k)^{-1}\succ 0$.
Since 
\begin{align*}
    &\tilde{H}_{S_k}(x_k) (S_k^T H(x_k) S_k)^{-1}\\
    &\hspace{0.5cm}= I +  (\tilde{H}_{S_k}(x_k) - S_k^T H(x_k) S_k)(S_k^T H(x_k) S_k)^{-1},
\end{align*}
positive definiteness of $\tilde{H}_{S_k}(x_k)$ is ensured if 
\begin{align}\label{eqn:Conditioned_S_K_appendix}
\|(\tilde{H}_{S_k}(x_k) - S_k^T H(x_k) S_k)(S_k^T H(x_k) S_k)^{-1}\|_2 < 1.   
\end{align}
Since 
\begin{align*}
    &\|(\tilde{H}_{S_k}(x_k) - S_k^T H(x_k) S_k)(S_k^T H(x_k) S_k)^{-1}\|_2 \\
    &\leq \|\tilde{H}_{S_k}(x_k) - S_k^T H(x_k) S_k \|_2 \| (S_k^T H(x_k) S_k)^{-1} \|_2 \\
    &\leq \|\tilde{H}_{S_k}(x_k) - S_k^T H(x_k) S_k \|_2 / \mu,
\end{align*}
a sufficient condition ensuring that Eq. \eqref{eqn:Conditioned_S_K_appendix} holds is
\begin{align*}
     \|(\tilde{H}_{S_k}(x_k) - S_k^T H(x_k) S_k) \|_2 < \mu.
\end{align*}

\section{Lemma 3}
To facilitate the analysis, we establish key error bounds due to finite difference estimations. 
\begin{lemma}\label{lemma:error_bounds}
Consider the ZO-RSN method \eqref{eqn:updaterule_ZO} for solving Problem \eqref{eqn:unconstrained_problem}. 
Let $e_k = \tilde{g}_{S_k}(x_k) - S_k^Tg(x_k)$ and $E_k =\tilde{H}_{S_k}(x_k) - S_k^T H(x_k) S_k$.
Then, 
\begin{align*}
    \| e_k \|_2  \leq  \sqrt{m}\alpha L_1/2,  \quad \text{and} \quad   \|E_k \|_2 \leq  5m\alpha L_2/3. 
\end{align*} 
\end{lemma}
\begin{proof}
Define $e_k = \tilde{g}_{S_k}(x_k) - S_k^Tg(x_k)$ and $E_k =\tilde{H}_{S_k}(x_k) - S_k^T H(x_k) S_k$. To prove the upper-bound for $\| e_k \|_2$, consider the first-order Taylor expansion of $f(x_k+\alpha s_{i,k})$ with the error term on Lagrange form: For $\theta \in [0,1]$ 
\begin{align*}
    & f(x_k+\alpha s_{i,k}) \\
    & = f(x_k) + \alpha g(x_k)^T s_{i,k} + \frac{\alpha^2}{2} s_{i,k} ^T \nabla^2 H(x_k+\theta \alpha s_{i,k}) s_{i,k}.
\end{align*}
Therefore, 
\begin{align*}
    \left |{\Delta_{\alpha s_{i,k}}f(x_k)}/{\alpha} -g(x_k)^T s_{i,k} \right|  
    &= \frac{\alpha}{2} \left | s_{i,k} ^T H(x_k+\theta \alpha s_{i,k}) s_{i,k} \right | \\
    &\leq \alpha \frac{L_1}{2},
\end{align*}
which implies a bound for each component of $e_k$. We can conclude that $\| e_k \|_2 \leq \sqrt{m}\alpha {L_1}/{2}$.

Next, denote the third derivative tensor of $f(x)$ by $f'''(x)$, where
\begin{equation*}
    [f'''(x)]_{ijk} = \frac{\partial}{\partial [x]_i}\frac{\partial}{\partial [x]_j}\frac{\partial}{\partial [x]_k} f(x).
\end{equation*}
For $u, v, w \in\mathbb{R}^d$, let 
$$
f'''(x)[u][v][w] = \sum_{i=1}^n \sum_{j=1}^n \sum_{k=1}^n [f'''(x)]_{ijk} [u]_i [v]_j [w]_k.
$$
The second-order Taylor expansion of $f(x_k+\alpha u)$ is
\begin{align*}
    f(x_k + \alpha u) & = f(x_k) + \alpha g(x_k)^T u + \frac{\alpha^2}{2} u^T H(x_k) u \\ & + \frac{\alpha^3}{6} f'''(x_k + \theta \alpha u)[u][u][u], \ \theta \in [0,1]. 
\end{align*}
The $L_2$-Lipschitz continuity assumption on $H(x)$ implies that $f'''(x)[u][v][w]$ is bounded by  $L_2\|u\|_2\|v\|_2\|w\|_2$. Thus, for $\theta_1,\theta_2,\theta_3 \in [0,1]$
\begin{equation*}
\begin{aligned}
    & \left |\Delta_{\alpha s_{i,k}}\Delta_{\alpha s_{j,k}} f(x_k)/{\alpha^2}- s_{i,k}^T H(x_k) s_{j,k} \right |  \\
    & = \frac{\alpha}{6} | f'''(c^1_k)[s_{i,k} +s_{j,k}][s_{i,k}+s_{j,k}][s_{i,k}+s_{j,k}] \\
    & \hspace{0.5cm}- f'''(c^2_k)[s_{i,k}][s_{i,k}][s_{i,k}] - f'''(c^3_k)[s_{j,k}][s_{j,k}][s_{j,k}]| \\
    & \leq \frac{5}{3} \alpha L_2, 
\end{aligned}
\end{equation*}
where $c^1_k = x_k + \theta_1 \alpha (s_{i,k}+s_{j,k})$, $c^2_k = x_k + \theta_2 \alpha s_{i,k}$ and $c^3_k = x_k + \theta_3 \alpha s_{j,k}$. This gives a bound on each component of $E_k$.
To finish the analysis, we invoke the following theorem:

\begin{theorem}[Ger\v{s}gorin theorem, \cite{horn2013matrix}]\label{thm:Gershgorin}
Let $A = [a_{ij}] \in M_n$ and let $R_i'(A) = \sum_{j \neq i} |a_{ij}|, \ i \in \{1,...,n\}$ denote the deleted absolute row sums of $A$. Consider the $n$ Ger\v{s}gorin discs $\{z \in \mathbf{C}:|z-a_i|\leq R_i'(A)\}$. The eigenvalues of $A$ are in the union of the Ger\v{s}gorin discs.
\end{theorem}
Finally, by applying Theorem~\ref{thm:Gershgorin} on $E_k^TE_k$, we have  $\rho(E_k^TE_k) \leq ({25}/{9})L_2^2m^2\alpha^2$. We can hence conclude that $\|E_k\|_2 = \sqrt{\rho(E_k^TE_k)} \leq ({5}/{3}) L_2m\alpha$. 

\end{proof}

\section{Proof of Theorem 1}
The RSN method chooses $x_{k+1}$ by minimizing the right hand side of \eqref{eqn:gammabound} with respect to $x$ and subject to the condition that $x - x_k$ is a linear combination of the columns of $S_k$. This constraint can directly be taken into account by the following change of variables to an $m$-dimensional variable vector $\lambda$:
\begin{equation*}
    x = x_k + \gamma S_k \lambda.
\end{equation*}
Denote $T_k(\lambda)$ as the upper bound of~\eqref{eqn:gammabound}, i.e.
\begin{equation*}
    T_k(\lambda) = f(x_k) + \gamma g(x_k)^TS_k \lambda + \frac{\gamma}{2} \|\lambda\|_{S_k^T H(x_k) S_k}^2.
\end{equation*}
Here, $\lambda_k = -(S_k^T H_k S_k)^\dagger S_k^T g(x_k)$ from Eq. \eqref{eqn:updaterule} is the $\lambda$ minimizing $T_k(\lambda)$.

Since the ZO-RSN method only accesses approximations of the sketched gradient and Hessian $\tilde{g}_{S_k}(x_k)$ and $\tilde{H}_{S_k}(x_k)$, it tries to minimize $\tilde{T}_k(\lambda)$, where
\begin{equation}\label{eqn:approx_gamma_bound}
    \tilde{T}_k(\lambda) = f(x_k) + \gamma \tilde{g}_{S_k}(x_k)^T\lambda + \frac{\gamma}{2}\|\lambda\|_{\tilde{H}_{S_k}(x_k)}^2. 
\end{equation}
Let $\tilde{\lambda}_k$ be the minimizer of $\tilde{T}_k(\lambda)$. By setting $x_{k+1} = x_k + \gamma S_k^T \tilde{\lambda}_k$, we get
\begin{equation}\label{eqn:ZO_RSN_upper_bound_no_alpha}
\begin{aligned}
    f(x_{k+1}) 
    & \leq T(\tilde{\lambda}_k) = T(\lambda_k) + T(\tilde{\lambda}_k) - T(\lambda_k) \\
    & = f(x_k) - \frac{\gamma}{2} \|g(x_k)\|^2_{S_k(S_k^TH(x_k)S_k)^{\dagger}S_k^T} \\ 
    & \hspace{0.5cm}+ \gamma g(x_k)^TS_k(\tilde{\lambda}_k - \lambda_k)  \\
    &\hspace{0.5cm}+ \dfrac{\gamma}{2}(\|\tilde{\lambda}_k\|_{S_k^TH(x_k)S_k}^2 - \|\lambda_k\|_{S_k^TH(x_k)S_k}^2) \\
    & = f(x_k) - \frac{\gamma}{2} \|g(x_k)\|_{S_k(S_k^TH(x_k)S_k)^{-1}S_k^T}^2 \\
    &\hspace{0.5cm}+ \gamma g(x_k)^TS_k(\tilde{\lambda}_k - \lambda_k) \\ 
    &\hspace{0.5cm} + \dfrac{\gamma}{2}(\tilde{\lambda}_k + \lambda_k)^TS_k^TH(x_k)S_k(\tilde{\lambda}_k-\lambda_k). 
\end{aligned}
\end{equation}
To complete the proof, we need to determine upper-bounds for $\|\lambda_k\|_2$, $\|\tilde{\lambda}_k - \lambda_k\|_2$ and $\|\lambda_k + \tilde{\lambda}_k\|_2$.
We first prove the upper-bound for $\|\lambda_k\|_2$.
Since $f(x)$ is $L_1$-smooth and $\mu$-strongly convex, $\mu I \preceq S_k^TH(x_k)S_k \preceq L_1 I$.
By the fact that $ S_k^TH(x_k)S_k \lambda_k = - S_k^T g(x_k)$, 
\begin{equation}\label{eqn:lambdabound}
    \|\lambda_k\|_2 = \|(S_k^T H(x_k) S_k)^{-1} S_k^T g(x_k) \|_2 \leq \|g(x_k)\|_2/\mu.
\end{equation}
We next find the upper-bound for $\|\tilde{\lambda}_k - \lambda_k\|_2$.
Define $e_k = \tilde{g}_{S_k}(x_k) - S_k^Tg(x_k)$ and $E_k =\tilde{H}_{S_k}(x_k) - S_k^T H(x_k) S_k$.
If $\alpha \leq 3\mu/(10L_2 m)$, then $\|E_k\|_2 \|(S_k^T H(x_k) S_k)^{-1}\|_2 \leq {1}/{2}$.
We can then use the following lemma:
\begin{lemma}[\cite{horn2013matrix}]\label{lemma:perturbation}
Let $A \in M_n$ be non-singular with condition number $\kappa(A)$, let $b,\Delta b \in \mathbb{R}^n$ and let $\Delta A \in M_n$ be such that $\|A^{-1}\|_2 \|\Delta A\|_2 < 1$. If $x = A^{-1} b$, there exits a $\Delta x$ such that 
\begin{equation*}
    (A+\Delta A) (x + \Delta x) = b + \Delta b,
\end{equation*}
and 
\begin{equation*}
    \|\Delta x\|_2 \leq \dfrac{\|A^{-1}\|_2}{1-\kappa(A) \frac{\|\Delta A\|_2}{\|A\|_2}}(\|\Delta b\|_2 + \|\Delta A\|_2 \|x\|_2),
\end{equation*}
\end{lemma}
By 
Lemma~\ref{lemma:perturbation}, and by the fact that $$\kappa(S_k^TH(x_k)S_k)/\|S_k^TH(x_k)S_k\|_2 = \|(S_k^TH(x_k)S_k)^{-1}\|_2,$$
we have
\begin{equation}\label{eqn:lambda_perturbation}
\begin{aligned}
    & \|\tilde{\lambda}_k - \lambda_k\|_2 \\
    & \leq \dfrac{\|(S_k^TH(x_k)S_k)^{-1}\|_2}{1-\kappa(S_k^TH(x_k)S_k) \frac{\|E_k\|_2}{\| S_k^TH(x_k)S_k\|_2}}(\|e_k\|_2 + \|E_k\|_2 \|\lambda_k\|_2)\\
    &\leq \frac{1}{\mu}(\sqrt{m} L_1 + \frac{10}{3} m L_2 \|\lambda_k\|_2)\alpha \\
    &\leq \frac{1}{\mu}(\sqrt{m} L_1 + m \frac{10 L_2}{3 \mu} \|g(x_k)\|_2)\alpha.
\end{aligned}
\end{equation}
We finally can prove the upper-bound for $\|\lambda_k + \tilde{\lambda}_k\|_2$:
\begin{align}
    \|\lambda_k + \tilde{\lambda}_k \|_2 
    &\leq 2\|\lambda_k \|_2 + \|\tilde{\lambda}_k - \lambda_k\|_2 \nonumber\\
    &\leq \frac{1}{\mu} \left (\sqrt{m} L_1 + \left( 2 + m \frac{10 L_2}{3 \mu} \right) \|g(x_k)\|_2 \right)\alpha. \label{eqn:lambdasum}
\end{align}
Next, plugging in inequalities \eqref{eqn:lambdabound}, \eqref{eqn:lambda_perturbation} and \eqref{eqn:lambdasum} into \eqref{eqn:ZO_RSN_upper_bound_no_alpha}, and then 
using the fact that $\|g(x_k)\|_2 \leq {(1 + \|g(x_k)\|_2^2)}/{2}$
\begin{align}
    f(x_{k+1}) 
    & \leq f(x_k) - \frac{\gamma}{2} \|g(x_k)\|_{S_k(S_k^TH(x_k)S_k)^{-1}S_k^T}^2 \nonumber \\
    &\hspace{0.5cm} + \alpha(C_1+C_2\alpha + \|g(x_k)\|_2^2 (C_1+C_3, \alpha))\label{eqn:recursion bound extended}
\end{align}
where $C_1 = \gamma(\sqrt{m}L + B)/(2\mu)$, $C_2 = \gamma L_1^2 [m + \sqrt{m}(1+B)]/(2\mu^2)$, $C_3 = \gamma L_1 [ \sqrt{m}L_1(1+B) + B(2+B) ] /(2\mu^2)$ and $B = 10mL_2/(3\mu)$.
Taking the expectation with respect to $x_k$ on both sides of Inequality~\eqref{eqn:recursion bound extended}, we have 
\begin{align}
    \mathbb{E}[f(x_{k+1})|x_k] 
    &\leq f(x_k) - \frac{\gamma}{2} \|g(x_k)\|^2_{G(x_k)}  \nonumber\\
    &\hspace{0.5cm}+ \alpha(C_1+C_2\alpha + \|g(x_k)\|_2^2 (C_1+C_3 \alpha)), \label{eqn:Expectation_bound}
\end{align}
where $G(x) = \mathbb{E}_{S_k \sim \mathcal{D}}[S_k(S_k^TH(x)S_k)^{-1}S_k^T]$. 

To prove the linear convergence of the ZO-RSN method from Eq. \eqref{eqn:Expectation_bound}, we need to bound $\|g(x_k)\|_2^2$ and $\|g(x_k)\|^2_{G(x_k)}$.
We first prove the upper bound for $\|g(x_k)\|_2$ by the $L_1$-smoothness assumption of $f(x)$, i.e. 
\begin{equation*}
    f(x) \leq f(y) + g(y)^T(x-y) + \frac{L_1}{2}\|x-y\|_2^2.
\end{equation*}
Setting $y=x_k$ and minimizing both sides with respect to $x$ separately results in 
\begin{equation}\label{eqn:Lsmoothbound}
    \|g(x_k)\|_2^2 \leq 2L_1(f(x_k) - f^\ast).
\end{equation}
We next show the lower bound for $\|g(x_k)\|^2_{G(x_k)}$. 
If $H(x_k)$ is non-singular, then
\begin{align*}
    \|g(x_k)\|_{G(x_k)}^2 
    & = \|H(x_k)^{\frac{1}{2}}H(x_k)^{-\frac{1}{2}} g(x_k)\|_{G(x_k)}^2 \\
    &\geq \rho \|g(x_k)\|_{H(x_k)^{-1}}^2.
\end{align*}
Setting $y=x_k$ in~\eqref{eqn:muhatbound}  and minimizing both sides of the equation with respect to $x$ separately gives
\begin{equation*}
    f^\ast \geq f(x_k) - \frac{1}{2\hat \mu} \|g(x_k)\|_{H(x_k)^{-1}}^2.
\end{equation*}
Therefore, 
\begin{equation}\label{eqn:strong_pesudoconvexity_bound}
     2 \rho \hat \mu (f(x_k)-f^\ast) \leq \rho \|g(x_k)\|_{H(x_k)^{-1}}^2 \leq \|g(x_k)\|_{G(x_k)}^2.
\end{equation}
Next, by plugging \eqref{eqn:Lsmoothbound} and \eqref{eqn:strong_pesudoconvexity_bound} into \eqref{eqn:Expectation_bound}, then by subtracting $f^\ast$ from both sides of the inequality, and after that by taking the total expectation, we get
\begin{equation*}
   V_{k+1} \leq [1-\rho \hat \mu \gamma  +  \alpha(C_1+\alpha C_3)] V_k + \alpha(C_1+C_2\alpha).
\end{equation*}
where $V_k = \mathbf{E}[f(x_{k})-f^\ast]$. 

If $\alpha$ satisfies $\alpha C_1 + \alpha^2 C_3 < \rho \hat \mu \gamma$, then by  applying the inequality recursively and by using the fact $\sum_{l=0}^{k-1} \beta^l \leq \sum_{l=0}^{\infty} \beta^l = 1/(1-\beta)$ for $\beta\in(0,1)$
%
\begin{equation}\label{eqn:powerbound}
V_{k}\leq (1-\rho \hat \mu \gamma +  \alpha(C_1+\alpha C_3))^kV_0+ \dfrac{\alpha(C_1+C_2\alpha)}{\rho \hat \mu \gamma -\alpha C_1 - \alpha^2 C_3}. 
\end{equation}
If 
$\alpha$ also satisfies 
\begin{equation}\label{eqn:deltatimesepsilon}
    \dfrac{\alpha(C_1+C_2\alpha)}{\rho \hat \mu \gamma -\alpha C_1 - \alpha^2 C_3} \leq \delta \varepsilon,
\end{equation}
where $\delta \in (0,1)$ and $\varepsilon$ is an expected sub-optimality, then the lower bound on the number of iterations follows. 

\section{Proof of Corollary 4.1}
To prove the result, we need to quantify $\rho$. This can be done by using the following lemma: 
\begin{lemma}{\cite{gower2019rsn}}\label{lemma:exactness}
If for all $x_k \in \mathbb{R}^n$ it holds with probability 1 that $Null(S_k^TH(x_k)S_k) = Null(S_k)$ and  $Range(H(x_k)) \subset Range(\mathbb{E}_{S_k \sim \mathcal{D}}[S_k^TS_k])$, then $\rho (x_k) = \lambda^+_{min}(\mathbb{E}_{S_k \sim \mathcal{D}}[\hat P(x_k)])$ which is positive, where
\begin{equation}
    \hat P(x_k) := H^{1/2}(x_k) S_k (S_k^TH(x_k)S_k)^{\dagger} S_k^T H^{1/2}(x_k).
\end{equation}
\end{lemma}
From this lemma, we can quantify $\rho$  by considering the cases when the columns of $S_k$ are chosen randomly without replacement from a basis of orthonormal eigenvectors of $H(x_k)$.
Let $\tilde{\Lambda}_{S_k}$ be an $m \times m$ diagonal matrix such that the eigenvalue corresponding to column $i$ of $S_k$ is the $i^{\text{th}}$ element on the diagonal of $\tilde{\Lambda}_{S_k}$ and let its square root be $\tilde{\Lambda}_{S_k}^{\frac{1}{2}}$. Then, 
\begin{align*}
    \hat P(x_k) 
    &= H^{1/2}(x_k) S_k (S_k^TH(x_k)S_k)^{-1} S_k^T H^{1/2}(x_k) \\
    & = S_k \tilde{\Lambda}_{S_k}^{\frac{1}{2}} (\tilde{\Lambda}_{S_k}^{\frac{1}{2}}S_k^TS_k\tilde{\Lambda}_{S_k}^{\frac{1}{2}})^{-1} \tilde{\Lambda}_{S_k}^{\frac{1}{2}}S_k^T = S_k S_k^T.
\end{align*}
The eigenvectors of all realizations of $\hat P(x_k)$ are the eigenvectors of $H(x_k)$, with eigenvalues $1$ for each vector that is among the columns of $S_k$ and eigenvalues 0 for the other vectors. Therefore, the orthonormal eigenvectors of $H(x_k)$ are also the eigenvectors of $\mathbb{E}_{S_k \sim \mathcal{D}}[\hat P(x_k)]$. 
%
Since the probability that this eigenvector is among the columns of $S_k$ is $m/n$, $v^T \mathbb{E}_{S_k \sim \mathcal{D}}[\hat P(x_k)] v = m/n$ for any eigenvector $v$. 
Thus, we can prove that $\rho = m/n$.

From Theorem \ref{thm:Convergence}, the iteration complexity bound can be approximated in Eq.\eqref{eqn:weaker_bound}. If we choose $\gamma = {1}/{\hat L}$  and some $\sigma \in (0,1)$ such that $\alpha = (\sqrt{C_1^2/4 + (1-\sigma)\rho \hat \mu \gamma} -{C_1}/{2} )/C_2$, then 
$$
\rho \hat \mu \gamma -\alpha C_1 - \alpha^2 C_3 = {\sigma m \hat \mu}/{(n \hat L)}. 
$$
Plugging this expression into Eq.\eqref{eqn:weaker_bound}, we complete the proof.

\end{document}